\documentclass{amsart}
\usepackage{amsmath,amssymb}
\usepackage{fix-cm,a4wide}
\usepackage{graphicx}

\newtheorem{theorem}{Theorem}[section] 
\newtheorem{lemma}[theorem]{Lemma}

\newtheorem{OB}[theorem]{Observation}

\theoremstyle{remark}
\newtheorem{ex}[theorem]{Example}

\theoremstyle{definition}
\newtheorem{defi}[theorem]{Definition}
\newtheorem{algo}[theorem]{Algorithm}

\newcommand{\tens}[1]{\langle #1 \rangle}
\newcommand{\cc}{\mathbb C}
\newcommand{\rr}{\mathbb R}
\newcommand{\cO}{\mathcal O}
\newcommand{\ff}{\mathbb F}
\DeclareMathOperator{\SL}{SL}
\DeclareMathOperator{\PSL}{PSL}
\DeclareMathOperator{\supp}{supp}
\DeclareMathOperator{\Core}{Core}
\DeclareMathOperator{\img}{img}

\begin{document}

\title[Search and test algorithms for TPP triples]{Search and test algorithms for\\Triple Product Property triples}
\author{Ivo Hedtke}
\address{Mathematical Institute, University of Jena, D-07737 Jena, Germany}
\thanks{The first author was supported by the Studienstiftung des Deutschen Volkes.}
\email{Ivo.Hedtke@uni-jena.de}
\author{Sandeep Murthy}
\address{29 Stephen Road, Oxford OX3 9AY, England}
   \email{sandeepr.murthy@gmail.com}


\maketitle

\begin{abstract}
In 2003 \textsc{Cohn} and \textsc{Umans} introduced a group-theoretic approach to fast matrix multiplication. This involves finding large subsets of a group $G$ satisfying the Triple Product Property (TPP) as a means to bound the exponent $\omega$ of matrix multiplication.

We present two new characterizations of the TPP, which are useful for theoretical considerations and for TPP test algorithms.  With this we describe all known TPP tests and implement them in \textsf{GAP} algorithms. We also compare their runtime. Furthermore we show that the search for subgroup TPP triples of nontrivial size in a nonabelian group can be restricted to the set of all nonnormal subgroups of that group. Finally we describe brute-force search algorithms for maximal subgroup and subset TPP triples. In addition we present the results of the subset brute-force search for all groups of order less than $25$ and selected results of the subgroup brute-force search for $2$-groups, $\SL_n\ff_q$ and $\PSL_2\ff_q$.
\end{abstract}

\section{Introduction}

\subsection{A Short History of Fast Matrix Multiplication}
\noindent The naive algorithm for matrix multiplication is an $\cO(n^3)$ algorithm. From \textsc{Volker Strassen} (see \cite{Strassen}) we know that there is an $\cO(n^{2.81})$ algorithm for this problem.
\textsc{Winograd} optimized \textsc{Strassen}'s algorithm. While the \textsc{Strassen-Winograd} algorithm is
the variant that is always implemented (for example in the famous GEMMW package),
there are faster ones (in theory) that are impractical to implement. The fastest
known algorithm runs in $\cO(n^{2.38})$ time (see \cite{Coppersmith} from \textsc{Coppersmith} and \textsc{Winograd}). Most researchers believe that an optimal
algorithm with $\cO(n^2)$ runtime exists, but since 1987 no further
progress was made in finding one. Because modern microprocessor architectures have complex memory hierarchies and increasing
parallelism, performance has become a complex tradeoff, not just a simple matter
of counting flops. Algorithms which make use of this technology were described in \cite{Alberto}. Another well known method is \emph{Tiling}: The normal algorithm can be speeded up
by a factor of two by using a six loop implementation that blocks submatrices so
that the data passes through the L1 Cache only once.

\subsection{The Exponent of the Matrix Multiplication}
\noindent Let $M(n)$ denote the number of field operations in characteristic $0$ required to multiply two $(n\times n)$ matrices. The exponent $\omega$ of the matrix multiplication is defined as
\begin{gather*}
\omega := \inf \{r \in \rr ~ : ~ M(n)=\cO(n^r)\}.
\end{gather*}
Details about the complexity of matrix multiplication and the exponent $\omega$ can be found in \cite{Burg}. The result from \textsc{Coppersmith} and \textsc{Winograd} says that $\omega < 2.38$.

\subsection{The Group-theoretic Approach from Cohn and Umans}
\noindent (The reader can find the necessary background on group- and representation theory in \cite{Alperin}, \cite{James}  and \cite{NeumannBook}. More details about the framework can be found in \cite{Cohn} and \cite{CohnZwei}.)

In 2003 \textsc{Cohn} and \textsc{Umans} introduced in \cite{Cohn} a group-theoretic approach to fast matrix multiplication. The main idea is to embed the matrix multiplication over a ring $R$ into the group ring $RG$, where $G$ is a (finite) group. A group $G$ admits such an embedding, if there are subsets $S$, $T$ and $U$ which fulfill the so-called \emph{Triple Product Property}. 

\begin{defi}[right quotient]
Let $G$ be a group and $\emptyset\neq X \subseteq G$ be a nonempty subset of $G$. The \emph{right quotient} $Q(X)$ of $X$ is defined by $Q(X):=\{xy^{-1} : x,y \in X\}$.
\end{defi}

Note that $Q(S)=S$ holds, if $S$ is a subgroup of $G$.

\begin{defi}[Triple Product Property]
We say that the nonempty subsets $S$, $T$ and $U$ of a group $G$ fulfill the \emph{Triple Product Property} (TPP) if for $s\in Q(S)$, $t\in Q(T)$ and $u \in Q(U)$,
$stu=1$ holds iff $s=t=u=1$.
\end{defi}

With $\tens{n,p,m}$ we denote the problem
\begin{gather*}
\tens{n,p,m} \colon \cc^{n\times p}\times \cc^{p\times m} \to \cc^{n\times m}, \quad (A,B)\mapsto AB,
\end{gather*}
of multiplying an $(n\times p)$ with a $(p\times m)$ matrix over $\cc$. We say that a group $G$ \emph{realizes} $\tens{s_1,s_2,s_3}$ if there are subsets $S_i\subseteq G$ of sizes $|S_i|=s_i$, which fulfill the TPP. In this case we call $(S_1,S_2,S_3)$ a \emph{TPP triple} of $G$.
Let us now focus on the embedding of the matrix multiplication into $\cc G$. Let $G$ realize $\tens{n,p,m}$ through the subsets $|S|=n$, $|T|=p$ and $|U|=m$. Let $A$ be an $(n\times p)$ and $B$ be a $(p\times m)$ matrix. We index the entries of $A$ and $B$ with the elements of $S$, $T$ and $U$ instead of numbers. Now we have
\[
(AB)_{s,u}=\sum\nolimits_{t \in T} A_{s,t}B_{t,u}.
\]
\textsc{Cohn} and \textsc{Umans} showed that this is the same as the coefficient of $s^{-1}u$ in the product
\begin{gather*}
\Big( \sum\nolimits_{s\in S, t\in T} A_{s,t}s^{-1}t\Big)
\Big( \sum\nolimits_{\hat t\in T,u\in U} B_{\hat t,u}\hat t^{-1}u\Big).
\end{gather*}
So we can read off the matrix product from the group ring product by looking at the coefficients of $s^{-1}u$ with $s\in S$ and $u\in U$.

\begin{defi}[TPP (subgroup) capacity]\label{def:TPPcap}
We define the \emph{TPP capacity} $\beta (G)$ of a nontrivial group $G$ as
$\beta(G) := \max\{npm : G\text{ realizes }\tens{n,p,m}\}$ and the \emph{TPP subgroup capacity} of $G$ as $\beta_\mathrm{g}(G) := \max\{npm : G\text{ realizes }\tens{n,p,m}\text{ through subgroups}\}$.
\end{defi}

Note that $\beta (G) \geq |G|$, because every group $G$ realizes $\tens{|G|,1,1}$ through the TPP triple $(G,1,1)$.

\begin{defi}[$r$-character capacity]\label{def:DG}
Let $G$ be a group with the character degrees $\{d_i\}$. We define the \emph{$r$-character capacity} of $G$ as $D_r(G):=\sum_i d_i^r$.
\end{defi}

We can now use $\beta$ and $D_r$ to get new bounds for $\omega$:

\begin{theorem}\textup{\cite[Thm. 4.1]{Cohn}}
If $G\neq 1$ is a finite group, then $\beta(G)^{\omega/3} \leq D_\omega (G)$.
\end{theorem}

Note that this leads to a nontrivial upper bound for $\omega$, iff $\beta (G) > D_3(G)$. Therefore, for a fixed group $G$ we search for TPP triples $(S,T,U)$ which maximize $|S|\cdot |T|\cdot |U|$, for example with a brute-force computer search. Note that the maximal $|S|\cdot |T|\cdot |U|$ equals to $\beta(G)$.

With their group theoretic framework \textsc{Cohn} and \textsc{Umans} were able to construct an algorithm for matrix multiplication with running time $\mathcal O(n^{2.41})$, see \cite[Sec. 6.3]{CohnZwei}.

From \textsc{Cohn} and \textsc{Umans} we know (\cite[Lem. 3.1.]{Cohn}), that $\beta(G)=|G|$ if $G$ is an abelian group. Therefore we only focus on \emph{nonabelian groups}. We also only focus on \emph{finite groups}.

\subsection{The Aim of this Work}
\noindent In  this article we present two new characterizations of the TPP, which are useful for theoretical considerations and for TPP test algorithms. With this we describe all known TPP tests and implement them in \textsf{GAP}. We also compare their runtime. Furthermore we show that the search for subgroup TPP triples of nontrivial size in a nonabelian group can be restricted to the set of all nonnormal subgroups of that group. Finally we describe brute-force search algorithms for maximal subgroup and subset TPP triples. In addition we present the results of the subset brute-force search for all groups of order less than $25$ and selected results of the subgroup brute-force search for $2$-groups, $\SL_n\ff_q$ and $\PSL_2\ff_q$.

\section{Basic Concepts}
\label{sec:concepts}
\noindent In this section we present facts that we use in our search algorithms. Details about the subset and subgroup search can be found in section \ref{sec:SearchAlgos}.\enlargethispage{\baselineskip}
Depending on the context, the symbol $1$ will denote either the number $1$, the group identity $1_G$, or the trivial subgroup $\{1_G\}$.

\begin{lemma}\textup{\cite[Lem. 2 and 3 and Thm. 4]{Hedtke}}
\begin{enumerate}
\item[(*)] Let $\emptyset \neq X\subseteq G$ be a nonempty subset of a group $G$ and $g\in G$. Then $1\in Q(X)$ and $g \in Q(X) \Leftrightarrow g^{-1} \in Q(X)$.
\item[(**)] If $S$, $T$ and $U$ fulfill the TPP then $Q(X)\cap Q(Y)=1$ holds for all $X\neq Y \in \{S,T,U\}$.
\item[(***)] If $(S,T,U)$ is a TPP triple with $1\in S\cap T \cap U$, then $S\cap T = T\cap U = S\cap U = 1$.
\end{enumerate}
\end{lemma}

\begin{lemma}\textup{\cite[Cor. 6]{Hedtke}} If $(S,T,U)$ is a TPP triple of $G$, then $|S|+|T|+|U|\leq |G|+2$.
\end{lemma}

First we note, that we are only interested in matrix-matrix multiplication, that means, we consider only TPP triples $(S,T,U)$ with $|S|$, $|T|$, $|U|>1$. Furthermore we have: Assume that in a TPP triple of $G$ one of $S$, $T$ or $U$ is $G$ itself. Then it follows from the lemma above, that the other two sets of the triple have size $1$. Because we omit this case, we only use:

\begin{OB}\label{ob:NotOneOrG}
It is sufficient to search TPP triples with $|S|$, $|T|$, $|U| \in \{2,\ldots,|G|-1\}$.
\end{OB}

From \textsc{Cohn} and \textsc{Umans} (see \cite[Lem.~2.1]{Cohn}) we know the following symmetry property:
\begin{lemma}
If $G$ realizes $\tens{n,p,m}$, then it does so for every permutation of $n$, $p$ and $m$.
\end{lemma}

We conclude the following:

\begin{OB}\label{ob:Order}
It is sufficient to search TPP triples with $|S| \geq |T| \geq |U|$.
\end{OB}

We know, that $\beta(G)\geq |G|$, therefore we are only interested in
\begin{OB}\label{ob:MinimumBeta}
It is sufficient to search TPP triples with $|S| \cdot |T| \cdot |U| > |G|$.
\end{OB}

Now we can combine the last two observations: We know that $|S| \cdot |T| \cdot |U| \leq |S|^3$. Therefore we have
\begin{OB}\label{ob:lower}
It is sufficient to search TPP triples with $|S| \geq \left\lceil\sqrt[3]{|G|}\right\rceil =: \ell (G)$,
where $\lceil x \rceil$ denotes the smallest integer not smaller than $x$.
\end{OB}
Assume the case, where we try to find triples that yield a nontrivial upper bound for $\omega$ instead of only finding nontrivial TPP triples. In this case we change the $|G|$ in the last two observations to $D_3(G)$.

\begin{defi}[basic TPP triple]
According to \textsc{Neumann} we call a TPP triple $(S,T,U)$ that fulfills $1\in S \cap T \cap U$ a \emph{basic TPP triple}.
\end{defi}

From \textsc{Neumann} we know the following facts that reduce the search space enormously:
\begin{lemma}\textup{\cite[Obs. 2.1]{Neumann}}\label{lemm:Neumann}
If $(S,T,U)$ is a TPP triple of $G$, then $(dSa,dTb,dUc)$ is a TPP triple for all $a,b,c,d \in G$, too.
\end{lemma}

Note that any TPP triple can be translated to a basic
TPP triple by lemma~\ref{lemm:Neumann}. Together with (***) it follows, that:

\begin{OB}\label{ob:Intersec}
It is sufficient to search TPP triples with $S \cap T = T \cap U = S \cap U = 1$.
\end{OB}

\begin{lemma}\textup{\cite[Obs. 3.1]{Neumann}}
If $(S,T,U)$ is a TPP triple, then $|S|(|T|+|U|-1)\leq |G|$, $|T|(|S|+|U|-1)\leq G$ and $|U|(|S|+|T|-1)\leq |G|$.
\end{lemma}
\begin{OB}\label{ob:Neumann}
It is sufficient to search TPP triples with $|S|(|T| + |U| - 1) \leq |G|$.
\end{OB}

\begin{ex} Let us focus on a brute-force search for TPP triples of subgroups in the \textsc{Mathieu} group $M_{11}$. We pick out this example because it is a big group, note that $|M_{11}|=7920$, with many (exactly $8651$) subgroups. Possible orders for $S$, $T$ or $U$ are (computed with \textsf{GAP}): \[\{ 1, 2, 3, 4, 5, 6, 8, 9, 10, 11, 12, 16, 18, 20, 24, 36, 48, 55, 60, 72, 120, 144, 360, 660, 720, 7920\}=:\mathcal P.\] Note that $\mathcal P$ has $26$ elements. If we use observation \ref{ob:NotOneOrG}, we define $\mathcal P := \mathcal P \setminus \{1,7920\}$, so we have $|\mathcal P|=24$. A naive idea would be $(|S|,|T|,|U|)\in \mathcal P^3 =: C$. Because $|C|=13824$ we would have a large search space for a brute-force search. Now we use observation~\ref{ob:Order} and achieve $|C|=2600$. If we only search for subgroup TPP triples that yield a nontrivial bound for $\omega$, we use observation~\ref{ob:MinimumBeta} with $D_3(G)=355208$ (again computed with \textsf{GAP}) instead of $|G|$. With this it follows that $|C|=404$. Finally we use \textsc{Neumann}s inequality from observation~\ref{ob:Neumann}, which yields to $|C|=0$. This example shows that the observations above reduce the search space for a brute-force search enormously. In this case it is not even necessary to start a search.
\end{ex}

\section{Two New Characterizations of the Triple Product Property}
\noindent In this section we present two new characterizations of the TPP. They are useful for theoretical considerations (in particular theorem \ref{lem:TPPold}) and for TPP test algorithms which we discuss in section~\ref{sec:TPPtests}.

\begin{theorem}\label{lem:TPPold}
Three subsets of $G$ form a basic TPP triple $(S,T,U)$ iff
\begin{gather*}
{\rm (i)} ~~ 1 \in S\cap T \cap U,\qquad
{\rm (ii)} ~~ Q(T) \cap Q(U) = 1\qquad\text{and}\qquad
{\rm (iii)} ~~ Q(S) \cap Q(T)Q(U) = 1.
\end{gather*}
\end{theorem}
This is not a limitation, because we only need to search for basic TPP triples.

\begin{proof}
First assume that $(S,T,U)$ is a basic TPP triple. (i) follows directly from the definition and (ii) from (**). Furthermore (*) implies that $1\in Q(S) \cap Q(T)Q(U)$. Now assume there is a common element $1 \neq x\in Q(S) \cap Q(T)Q(U)$.  Then $x=s=tu$, for some $s\in Q(S)$ and $tu\in Q(T)Q(U)$. This means $1=x^{-1}x=s^{-1}tu$, but the TPP for $(S,T,U)$ implies that $1=s^{-1}=t=u$ and therefore $x=1$, a contradiction.

Now assume that the equations (i)--(iii) hold for the subsets $S$, $T$ and $U$. Consider
the triple quotient product $stu$ for arbitrary
elements $s\in Q(S)$, $t\in Q(T)$ and
$u\in Q(U)$. Then $stu = 1$ is equivalent to $s^{-1}=tu$. Now (iii) implies that $s^{-1}=tu=1$ and (*) together with (ii) imply that $t=u=1$ and so $(S,T,U)$ is a basic TPP triple.
\end{proof}

\begin{defi}[subtransversal, support]
Let $C$ be a finite nonempty set and $\mathcal{C}=\{C_{1},\ldots,C_{k}\}$ a partition of it.
A set $X \subseteq C$ is called a \emph{subtransversal} for $\mathcal{C}$ with \emph{support}
$\supp_{\mathcal{C}}(X) = \mathcal{T} \subseteq \mathcal{C}$ if for all $C_i \in \mathcal{C}$
\[
\left|  X \cap C_i \right| = \begin{cases} 1 & C_i \in \mathcal{T}, \\ 0 & \text{otherwise.} \end{cases}
\]
It then follows that $\left| X \right| = \left| \mathcal{T} \right|$. In the special
case when the collection $\mathcal{C}$ is the set of left (or right) cosets of a subgroup
$S$ of a group $G$, then any subtransversal $T$ for $G/S$ (or $S \setminus G$) will simply be called
a subtransversal for $S$ in $G$.
\end{defi}

\begin{theorem}\label{th:subtrans}
Let $G$ be a group, $S$ a subgroup of $G$, and $T$, $U$ subsets of $G$.
\begin{enumerate}
\item If $(S,T,U)$ is a basic TPP triple of G then $T$ and $U$ are subtransversals
                         for $S$ in $G$ such that
\begin{gather}\label{eq:Green}
\supp_{S\setminus G}(T) \cap \supp_{S\setminus G}(U) = \{S\}.
\end{gather}
\item If $T$ and $U$ are also subgroups of $G$, and $T$ and $U$ are
                         subtransversals for $S$ in $G$ satisfying \eqref{eq:Green} then $(S,T,U)$ is
                         a TPP triple of $G$.
\end{enumerate}
\end{theorem}

\begin{proof}
(i) We have $S\cap T=S\cap U=1$, by (***).
Let $S\neq Sr\in S\setminus G$ be a nontrivial coset of $S$.
Assume distinct elements $1\neq t,t'\in T \cap Sr$, where $t=sr$ and $t'=s'r$ for distinct $s,s'\in S$. 
Then $1\neq t't^{-1}=s's^{-1}\in S\cap Q(T)$, and that contradicts the TPP requirement (**). So for any coset $Sr\in S\setminus G$, we have $|T\cap Sr|\leq1$. The same holds for $U$.
Therefore $T$ and $U$ are subtransversals for $S\setminus G$.
Now assume distinct elements $1\neq t=sr\in T\cap Sr$ and $1\neq u=s'r\in U\cap Sr$ for a nontrivial coset $Sr\neq S$.
Then $1\neq tu^{-1}=s(s')^{-1} \in Q(S)\cap Q(T)Q(U) = S\cap Q(T)Q(U)$ and that contradicts theorem \ref{lem:TPPold}. This shows that $T$ and $U$ fulfill \eqref{eq:Green}.

(ii) Assume $T$ and $U$ are subtransversals of $S\setminus G$ and fulfill \eqref{eq:Green}. Since $S$, $T$ and $U$ are subgroups we have $1 \in S \cap T \cap U$ and $1\in S\cap TU$. Because the intersection of the supports of $T$ and $U$ is $\{S\}$, $T\cap U=1$ holds. Now assume that there is an $1\neq x \in S\cap TU$. So there are $s\in S$, $t\in T$ and $u\in U$ with $1\neq s = tu$. Therefore we have $tu\in S$, which is equivalent to $t\in Su^{-1}$. But then would be $t=u=1$, a contradiction. So $(S,T,U)$ is a TPP triple by theorem \ref{lem:TPPold}.
\end{proof}

\begin{OB}
Three subgroups $S$, $T$, $U$ of a group $G$ form a TPP triple iff $T$ and $U$
are subtransversals for $S$ in $G$, satisfying \eqref{eq:Green}.
\end{OB}

\begin{theorem}
Let $G$ be a group.
If $(S,T,U)$ is a TPP triple of subgroups where at least one of $S$, $T$ or $U$ is normal in $G$, then $|S|\cdot |T| \cdot |U| \leq |G|$.
\end{theorem}

\begin{proof}
Without loss of generality assume that $S$ is nontrivial, proper and normal in $G$. Because $S$ is normal in $G$, we have $S\setminus G = G/S$. We will work with $G/S$. From theorem~\ref{th:subtrans} we know, that $T$ and $U$ are subtransversals for $G/S$ that fulfill \eqref{eq:Green}.
Let $S_1:=S,S_2,\ldots,S_v$ be the $v=[G:S]\geq 2$ elements of $G/S$. We define $\mathcal T := \supp_{G/S}(T)\subseteq G/S$ and $\mathcal U$ in the same way.
Then $\mathcal T$ and $\mathcal U$ are subgroups of $G/S$ with $|\mathcal T|=|T|$, $|\mathcal U|=|U|$ and $\mathcal T \cap \mathcal U=1$.
To see this, first consider $\mathcal T$.
Because $T$ and $U$ are subtransversals for $G/S$ that fulfill \eqref{eq:Green}, we have $1_{G/S}=S\in \mathcal T$. 
Now let $rS, r'S\in \mathcal T$ be cosets of $S$. So there are $t,t'\in T$ such that $T\cap rS=\{t=rs\}$ and $T\cap r'S=\{t'=r's'\}$, where $s,s'\in S$. They have the product $rsr's'=tt'\in T \cap rr'S$ and so $|T\cap rr'S| \geq 1$. Since $T$ is a subtransversal for $G/S$ it follows that $|T\cap rr'S| = 1$ and so $rr'S\in \mathcal T$. The inverse of any given $rS\in \mathcal T$ is $r^{-1}S\in G/S$ and if $t\in T$ is such that $T\cap rS=\{t\}$
then $t=rs$ for some $s\in S$, and $t^{-1}=s^{-1}r^{-1}\in Sr^{-1}=r^{-1}S$
(because $S$ is normal), and so $t^{-1}\in T\cap r^{-1}S$,
and we can conclude that $r^{-1}S\in\mathcal{T}$ as well. The same holds for $\mathcal U$.
Because \eqref{eq:Green} we know that $\mathcal T \cap \mathcal U=1$ and so $(\mathcal T, \mathcal U,1)$ is a TPP triple of $G/S$ by theorem~\ref{lem:TPPold}. From \textsc{Neumann}s inequality (see lemma \ref{lemm:Neumann}) it follows that $|\mathcal T|(|\mathcal U|+1-1) = |\mathcal T|\cdot |\mathcal U| = |T| \cdot |U| \leq |G/S|= |G|/|S|$, which we wanted to show.
\end{proof}

We can use the result above to create an additional filter for the search space of subgroup TPP triples:

\begin{OB}\label{OB:NONNORMAL}
If we are only interested in subgroup TPP triples of nontrivial size it is sufficient to search for $S$, $T$ and $U$ that are all nonnormal.
\end{OB}

\section{On the TPP Capacity of Nonabelian Groups}
\noindent In this section we present some facts about the TPP capacity of nonabelian groups.
The following two results state nontrivial lower bounds for $\beta(G)$.
The ideas behind the proofs are those of the authors, but the proofs
as formulated here are those of \textsc{Neumann}.

\begin{lemma}\label{lem:NonAB}
If $G$ is a nonabelian group with a nonnormal subgroup $S$ of index $[G:S]=3$, then $\beta(G) \geq \frac43 |G|$.
\end{lemma}

\begin{proof}
Let $S\setminus G$ be the right coset space of $G$ of size $[G:S]=3$. There is a natural homomorphism $\phi \colon G \to S_3$, defined by $g\mapsto \pi_g$ for all $g\in G$, where $\pi_g$ describes a permutation action $(g,Sr)\mapsto Srg=Sr_g \in S\setminus G$ for all $Sr\in S\setminus G$, of $g$ on $S\setminus G$. The homomorphism $\phi$ has a kernel $K:=\ker\phi \unlhd G$, which is $\Core_G(S)$ (the largest normal subgroup of $G$ contained in $S$). The quotient group $G/K\cong \img \phi \leq S_3$ is isomorphic to a transitive subgroup in $S_3$. The nontrivial transitive subgroups of $S_3$ are $A_3$ and $S_3$ itself. Since $S$ is nonnormal, $K < S$ is proper and so $|G/K| > 3$. Thus $G/K\cong S_3$. From \cite[Lem. 2.2]{Cohn} we know that $\beta(G) \geq \beta(G/K) \beta(K)$. It follows (see table \ref{tab:small}), that $\beta(G) \geq 8\beta(K) \geq 8|K|=8|G|/|G/K|=\frac86|G|=\frac43|G|$.
\end{proof}

\begin{lemma}\label{lem:NonAB2}
If $G$ is a nonabelian group with a self-normalising subgroup $S$ of index $[G:S]=4$, then $\beta(G) \geq \frac32 |G|$.
\end{lemma}

\begin{proof}
The subgroup $S$ has a normal core $K:=\Core_G(S)\unlhd G$ which is proper $K < S$, such that the quotient group $G/K$ is of order $|G/K|>|G/S|=4$. The core $K$ is the kernel of the natural homomorphism $\phi\colon G\to S_4$ which describes the permutation action of $G$ on $G/S$, and so $G/K\cong \img \phi \leq S_4$ is a transitive subgroup of order greater than $4$. The only possibilities are $A_4$, $D_8$ and $S_4$. Because we assume that $S$ is self-normalising, the case $G/K\cong D_8$ is not feasible. The statement follows from $\beta(A_4)/|A_4|=\beta(S_4)/|S_4|=3/2$ (see table \ref{tab:small}) like in the proof of lemma \ref{lem:NonAB}.
\end{proof}

\section{Triple Product Property Test Algorithms}
\label{sec:TPPtests}

\noindent At the moment there are five algorithms to test the TPP. In this section we will present them and compare their running time. 

We start with two algorithms that directly came from the TPP definition. We define $L:=Q(S)Q(T)Q(U)$ as a list, not as a set. Then we count, how many $1$'s are in it. The TPP is fulfilled iff there is only one $1$ in $L$.
\begin{algo}\texttt{TPPTestNaiv( S, T, U )}

\hrule
\noindent\verb|OUTPUT: TPP fulfilled: true / false|

\noindent\verb|  |$\mathtt{L:=Q(S) Q(T) Q(U)}$\\
\verb|  if ( #{ |$\mathtt{1_G\in L}$\verb| } |$\mathtt{>1}$\verb| ) then|\\
\verb|    return false;|\\
\verb|  fi;|\\
\verb|  return true;|\hrule
\end{algo}
The algorithm above is very naive, because we use the \emph{complete} list $L$. This needs a lot of time and memory. It is a better idea to search element-wise for $1$'s.
\begin{algo}\label{algo:TPPTest}
\texttt{TPPTest( S, T, U )}

\hrule
\noindent\verb|OUTPUT: TPP fulfilled: true / false|

\noindent\verb|  |$\mathtt{Q_S:=Q(S)}$\verb|, |$\mathtt{Q_T:=Q(T)}$\verb|, |$\mathtt{Q_U:=Q(U)}$\\
\verb|  for |$\mathtt{s \in Q_S}$\verb| do|\\
\verb|    for |$\mathtt{t \in Q_T}$\verb| do|\\
\verb|      for |$\mathtt{u \in Q_U}$\verb| do|\\
\verb|        if( |$\mathtt{stu=1}$\verb| and ( |$\mathtt{s\neq 1}$\verb| or |$\mathtt{t\neq 1}$\verb| or |$\mathtt{u\neq 1}$\verb| ) ) then|\\
\verb|          return false;|\\
\verb|        fi;|\\
\verb|  od; od; od;|\\
\verb|  return true;|\hrule
\end{algo}

From \textsc{Hendrik Orem} we know the following equivalent form of the TPP:
\begin{lemma}\textup{\cite[Thm. 2.1]{Orem}}
Subsets $S$, $T$ and $U$ of $G$ satisfy the TPP iff
\[
|S^{-1}|\cdot |U|=|S^{-1}U| \qquad\text{and}\qquad (S^{-1}(Q(T)\setminus 1) U) \cap S^{-1}U = \emptyset.
\]
\end{lemma}

Note that for a TPP test based on the lemma above we only need to compute one of the right quotients, instead of all three right quotients in the original TPP definition.

\begin{algo}\label{algo:TPPTestOrem}
\texttt{TPPTestOrem( S, T, U )}

\hrule
\noindent\verb|OUTPUT: TPP fulfilled: true / false|

\noindent\verb|  |$\mathtt{S_i := S^{-1}}$\verb|, |$\mathtt{S_{iu} := S_i U}$\\
\verb|  if( |$\mathtt{|S_i|\cdot |U|=|S_{iu}|}$\verb| ) then|\\
\verb|    if( |$\mathtt{S_i(Q(T) \setminus 1)U \cap S_{iu} = \emptyset}$\verb| ) then|\\
\verb|      return true;|\\
\verb|  fi; fi;|\\
\verb|  return false;|\hrule
\end{algo}

The fourth algorithm comes from the TPP reformulation of theorem \ref{lem:TPPold}, it requires a basic triple ($1\in S\cap T\cap U$) as input:

\begin{algo}\label{algo:TPPTestMurthy}
\texttt{TPPTestMurthy( S, T, U )}

\hrule
\noindent\verb|OUTPUT: TPP fulfilled: true / false|

\noindent\verb|  |$\mathtt{Q_T := Q(T)}$\verb|, |$\mathtt{Q_U := Q(U)}$\\
\verb|  if( |$\mathtt{Q_T \cap Q_U = 1}$\verb| ) then|\\
\verb|    |$\mathtt{Q_S := Q(S)}$\\
\verb|    if( |$\mathtt{Q_S \cap Q_TQ_U = 1}$\verb| ) then|\\
\verb|      return true;|\\
\verb|  fi; fi;|\\
\verb|  return false;|\hrule
\end{algo}

Now we focus on TPP tests for subgroups. We start with the TPP test inspired by theorem \ref{th:subtrans} about subtransversals.

\begin{algo}\texttt{TPPTestMurthyHedtkeGRP( S, T, U )}

\hrule
\noindent\verb|OUTPUT: TPP fulfilled: true / false|

\noindent\verb|  if( |$\mathtt{S \cap T = 1}$\verb| and |$\mathtt{S \cap U = 1}$\verb| ) then|\\
	\verb|    for |$\mathtt{X \in (S\setminus G)\setminus \{S\}}$\verb| do|\\
\verb|      if( |$\mathtt{|X \cap T| + |X \cap U| > 1}$\verb| ) then|\\
\verb|        return false;|\\
\verb|    fi; od;|\\
\verb|  else|\\
\verb|    return false;|\\
\verb|  fi;|\\
\verb|  return true;|\hrule
\end{algo}

In the case where $S$, $T$ and $U$ are subgroups, we have $Q(S)=S$, $Q(T)=T$ and $Q(U)=U$. Therefore the remaining test algorithms for subgroups are:
\begin{algo}\texttt{TPPTestNaivGRP( S, T, U )}
	
	\hrule
\noindent\verb|OUTPUT: TPP fulfilled: true / false|
	
\noindent\verb|  |$\mathtt{L:=S \cdot T \cdot U}$\\
	\verb|  if ( #{ |$\mathtt{1_G\in L}$\verb| } |$\mathtt{>1}$\verb| ) then|\\
	\verb|    return false;|\\
	\verb|  fi;|\\
	\verb|  return true;|\hrule
\end{algo}

\begin{algo}\texttt{TPPTestGRP( S, T, U )}
	
	\hrule
\noindent\verb|OUTPUT: TPP fulfilled: true / false|
	
\noindent\verb|  for |$\mathtt{s \in S}$\verb| do|\\
	\verb|    for |$\mathtt{t \in T}$\verb| do|\\
	\verb|      for |$\mathtt{u \in U}$\verb| do|\\
	\verb|        if( |$\mathtt{stu=1}$\verb| and ( |$\mathtt{s\neq 1}$\verb| or |$\mathtt{t\neq 1}$\verb| or |$\mathtt{u\neq 1}$\verb| ) ) then|\\
	\verb|          return false;|\\
	\verb|        fi;|\\
	\verb|  od; od; od;|\\
	\verb|  return true;|\hrule
\end{algo}

\begin{algo}\texttt{TPPTestOremGRP( S, T, U )}
	
	\hrule
\noindent\verb|OUTPUT: TPP fulfilled: true / false|
	
\noindent\verb|  |$\mathtt{S_{u} := S \cdot U}$\\
\verb|  if( |$\mathtt{|S|\cdot |U|=|S_{u}|}$\verb| ) then|\\
\verb|    if( |$\mathtt{S(T \setminus 1)U \cap S_{u} = \emptyset}$\verb| ) then|\\
\verb|      return true;|\\
\verb|  fi; fi;|\\
\verb|  return false;|\hrule
\end{algo}

\begin{algo}\texttt{TPPTestMurthyGRP( S, T, U )}
	
	\hrule
\noindent\verb|OUTPUT: TPP fulfilled: true / false|
	
\noindent\verb|  if( |$\mathtt{T \cap U = 1}$\verb| ) then|\\
	\verb|    if( |$\mathtt{S \cap T\cdot U = 1}$\verb| ) then|\\
	\verb|      return true;|\\
	\verb|  fi; fi;|\\
	\verb|  return false;|\hrule
\end{algo}

Now we compare the running time of the different TPP tests. For this we use the search algorithms described in the next section with the different test algorithms. We use a \textsf{GAP} (\cite{GAP}) implementation of the algorithms in this paper with the SONATA package (\cite{Sonata}) for the \texttt{Subgroups} routine. The interested reader can get the \textsf{GAP} codes from the first author via email.

For the subgroup test algorithms we run the brute-force search algorithm $10$ times (except for the groups of order $64$ and $96$, because there are more than $200$ such groups) for each nonabelian group of order less than $128$. The results in figure~\ref{fig:SubGroupAll} show the mean \emph{cumulative} running time of the search algorithm for all nonabelian groups \emph{up to} a given order. We recommend to use \texttt{TPPTestMurthyGRP}.

For the subset test algorithms we do the same, but only up to group order $20$ and only with 2 repeats to build the mean value. The results are shown in figure \ref{fig:SmallAll} (left). It is very obvious, that compared to the naive TPP test the algorithms \ref{algo:TPPTest}, \ref{algo:TPPTestOrem} and \ref{algo:TPPTestMurthy} are very fast. Therefore figure \ref{fig:SmallAll} (right) shows the details. Again, we recommend to use \texttt{TPPTestMurthy}.
\begin{figure}
\centering
\includegraphics[trim=0.7cm 9cm 1cm 9.6cm, clip, height=6.5cm]{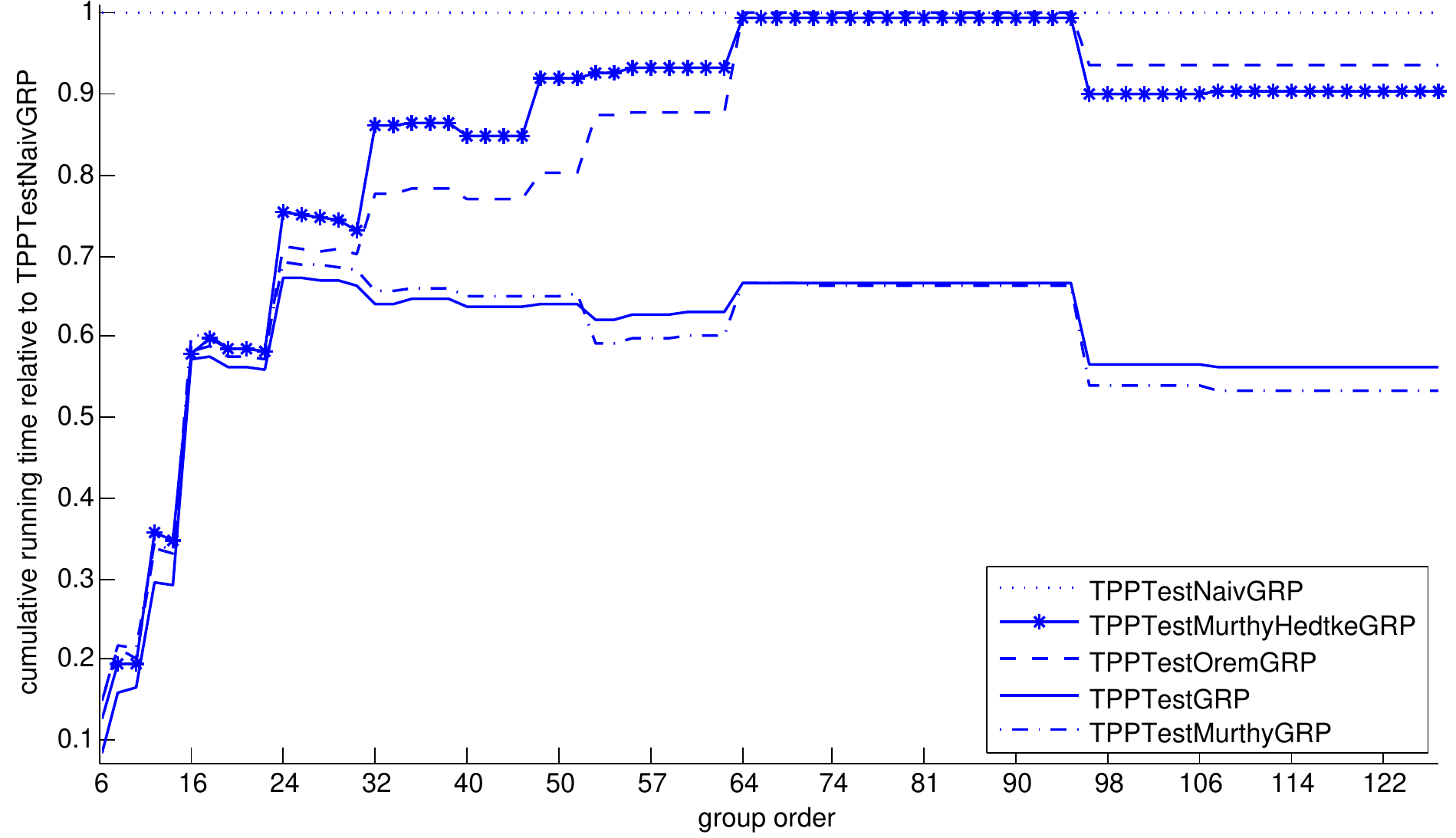}
\caption{\emph{Cumulative} running time \emph{relative} to \texttt{TPPTestNaivGRP} of the brute-force search for \emph{subgroups} with the five different test algorithms. (The vertical
axis is measuring the cumulative runtime of a sequence of 10 runs for
each of the five different TPP verification algorithms, averaged by 10, so
the closer the lines are to the zero line, the better.)}
\label{fig:SubGroupAll}
\end{figure}

\section{Brute-force Search}
\label{sec:SearchAlgos}
\begin{figure}[b]
\centering
\includegraphics[trim=4.6cm 9.7cm 5.3cm 10.2cm, clip, height=6cm]{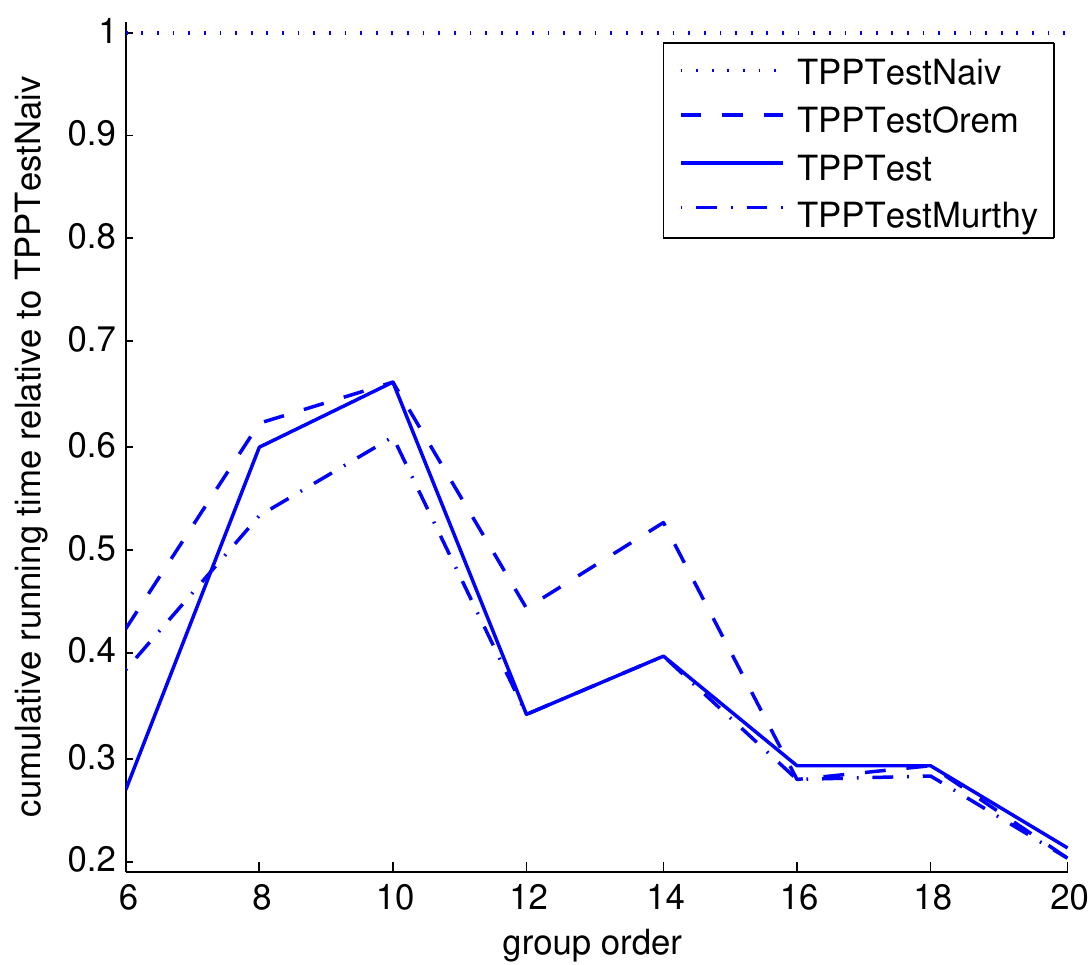}~~~~
\includegraphics[trim=5.2cm 10.3cm 6cm 10.8cm, clip, height=6cm]{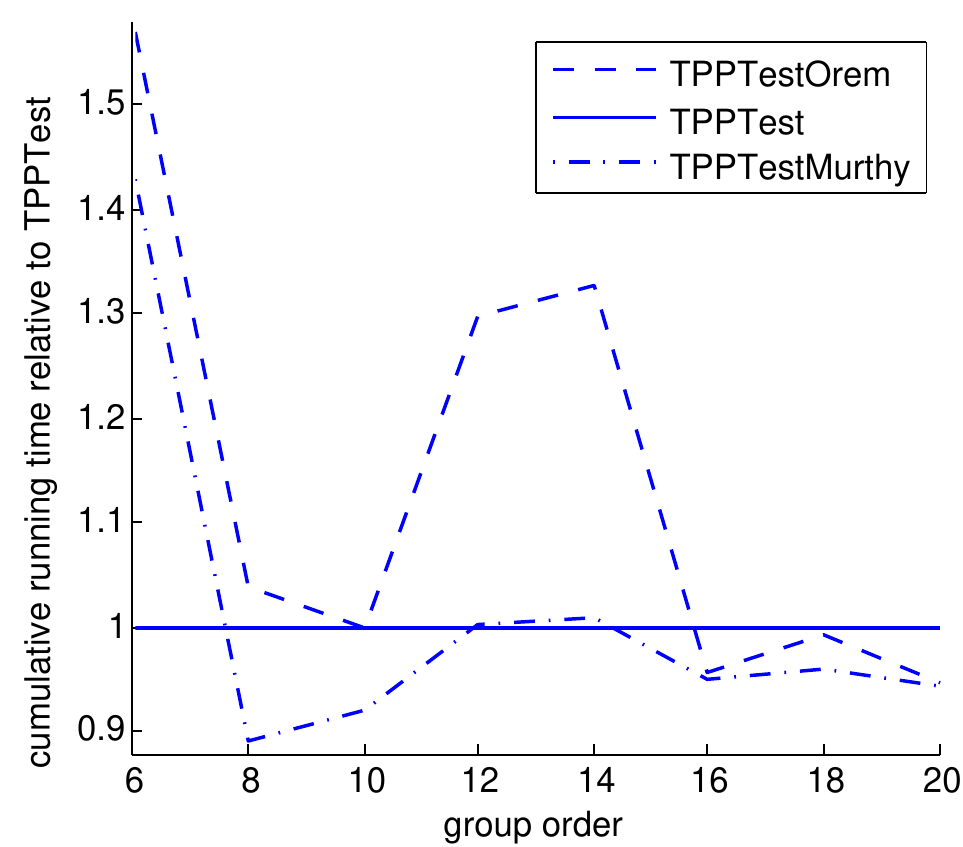}
\caption{Left: \emph{Cumulative} running time \emph{relative} to \texttt{TPPTestNaiv} of the brute-force search for \emph{subsets} with the four different test algorithms. Right: \emph{Cumulative} running time of \texttt{TPPTestOrem} and \texttt{TPPTestMurthy} \emph{relative} to \texttt{TPPTest}.}
\label{fig:SmallAll}
\end{figure}

\noindent In this section we describe the brute-force search algorithms that we use for our computations.
We use the observations from section \ref{sec:concepts} to reduce the search space for $S,T,U\subseteq G$ and the TPP tests from the last section.
Although it is relatively quick to search for subgroup TPP triples,
they do not fully describe the TPP capacity of a group, as the 
following theorem shows.

\begin{theorem}
For every finite group $G$, $\beta(G) \geq \beta_\mathrm{g}(G)$ holds.
There are groups with $\beta(G) > \beta_\mathrm{g}(G)$.
\end{theorem}

\begin{proof}
The first statement is trivial, because the search space for $\beta$ includes the one for $\beta_\mathrm{g}$.
For the second statement consider the group $D_{10}=\langle d,s : s^2=d^5=1, sds=d^{-1} \rangle$ of order $10$. From table \ref{tab:small} we know that $\beta_\mathrm{g}(D_{10})=10$. But the subsets $S:=\langle s \rangle$, $T:=\{d,s\}$ and $U:=\{1,sd,d^3\}$ of $D_{10}$ realize the problem $\tens{2,2,3}$ of size $12$, and so $\beta(D_{10})=12$.
\end{proof}

From the result above we see, that it is necessary so search for \emph{subset} TPP triples. Therefore we present search algorithms for either subgroup TPP triples and for subset TPP triples.

\subsection{Subgroup Search}

\noindent We use the following algorithm to search for subgroup TPP triples in a given nonabelian group.
Note that we use the \texttt{Subgroups} routine from the SONATA package and that the subgroups returned from that command are ordered ascending by the sizes of the subgroups.
Our command \texttt{NonnormalSubgroups} is only a filter for the nonnormal \texttt{Subgroups} of $G$.

\begin{algo}\texttt{TestGRP( G )}

\hrule
\noindent\verb|OUTPUT: subgroups |$\mathtt{S_{max}}$\verb|, |$\mathtt{T_{max}}$\verb|, |$\mathtt{U_{max}}$\verb| that realize a problem of size |$\beta_\mathrm{g}$\verb|(G)|

\noindent\verb| 1  |$\mathtt{n_{max}}$\verb+:=|G|, +$\mathtt{S_{max}}$\verb+:=G, +$\mathtt{T_{max}}$\verb+:={+$\mathtt{1_G}$\verb+}, +$\mathtt{U_{max}}$\verb+:={+$\mathtt{1_G}$\verb+}+\\
\verb+ 2  C:=NonnormalSubgroups(G); a:=|C|;+\hfill\emph{(observation \ref{OB:NONNORMAL})}\\
\verb# 3  lastS:=| { c#$\in$\verb#C : |c| < #$\ell(\mathtt{G})$\verb# } | + 1;#\hfill\emph{(observation \ref{ob:lower})}\\
\verb+ 4  for i in [a, a-1, ..., lastS] do+\hfill\emph{(observation \ref{ob:lower})}\\
\verb+ 5    S:=C[i];+\\
\verb+ 6    for j in [i-1, i-2, ..., 1] do+\hfill\emph{(observation \ref{ob:Order})}\\
\verb+ 7      T:=C[j];+\\
\verb+ 8      if ( S +$\cap$\verb+ T = 1 ) then+\hfill\emph{(observation \ref{ob:Intersec})}\\
\verb+ 9        for k in [j-1, j-2, ..., 1] do+\hfill\emph{(observation \ref{ob:Order})}\\
\verb+10          U:=C[k];+\\
\verb#11          if ( |S|#$\mathtt \cdot$\verb#(|T| + |U| - 1) #$\leq$\verb+ |G| ) then+\hfill\emph{(observation \ref{ob:Neumann})}\\
\verb#12            if ( |S|#$\mathtt \cdot$\verb#|T|#$\mathtt \cdot$\verb#|U| > #$\mathtt{n_{max}}$\verb+ ) then+\hfill\emph{(observation \ref{ob:MinimumBeta})}\\
\verb+13              if ( T +$\cap$\verb+ U = 1 and S +$\cap$\verb+ U = 1) then+\hfill\emph{(observation \ref{ob:Intersec})}\\
\verb+14                if TPPfulfilled(S,T,U) then+\\
\verb+15                  +$\mathtt{n_{max}}$\verb+:=|S|+$\mathtt \cdot$\verb+|T|+$\mathtt \cdot$\verb+|U|; +$\mathtt{S_{max}}$\verb+:=S; +$\mathtt{T_{max}}$\verb+:=T; +$\mathtt{U_{max}}$\verb+:=U; +\\
\verb+16                  break;+\\
\verb+17              fi; fi;+\\
\verb+18            else+\\
\verb+19              break;+\\
\verb+20  fi; fi; od; fi; od; od; return +$\mathtt{S_{max}}$\verb|, |$\mathtt{T_{max}}$\verb|, |$\mathtt{U_{max}}$\verb+;+\hrule
\end{algo}

For the command \texttt{TPPfulfilled} in line 14 of the code we use one of the TPP tests for subgroups described in the previous section. The \texttt{break} command in line 16 is used because all other $U$ that occur in the inner \texttt{for} loop have a size smaller or equal than $U_\mathrm{max}$ and we don't need to test them, because we will not get a bigger TPP triple. The \texttt{break} command in line $19$ is used because the multiplicative size of all other TPP triples with an $U$ from the inner loop have a size smaller than $n_\mathrm{max}$.

\subsection{Subset Search}

\noindent The search for subset TPP triples is very similar to the subgroup search. But instead of the \texttt{Subgroups} routine we use a method to generate subsets of $G$. We followed an idea of \textsc{E. Burnett} (see \cite{Burnett}) to generate the subsets on the fly in our code. (We also had the idea to generate a binary representation of the subsets. Each subset of $G$ can be identified as an element of $\{0,1\}^{|G|}$. But it needs too much memory to generate the whole object $\{0,1\}^{|G|}$ and therefore we are glad, that we found a way to generate the subsets one after another.)

\begin{algo} \textsc{GAP} implementation of the \texttt{GenerateSubset} method from \textsc{E. Burnett}\hrule

\noindent\verb#GenerateSubset := function ( OrderG , SizeS , NumberS )#\\
\verb#  local BoolVector, Offset, NumberZeros, Low, High;#\\~\\
\verb#  BoolVector:=[]; Offset:=0; NumberS:=NumberS-1;#\\
\verb#  #\\
\verb#  while true do#\\
\verb#    NumberZeros:=0; Low:=0; High:=Binomial(OrderG-1, SizeS-1);#\\
\verb#    #\\
\verb#    while (NumberS >= High) do#\\
\verb#      NumberZeros:=NumberZeros+1; Low:=High;#\\
\verb#      High:=High+Binomial(OrderG-NumberZeros-1, SizeS-1);#\\
\verb#    od;#\\
\verb#    #\\
\verb#    Add(BoolVector, Offset+NumberZeros+1);#\\
\verb#    #\\
\verb#    if (SizeS = 1) then#\\
\verb#      return BoolVector;#\\
\verb#    else#\\
\verb#      OrderG:=OrderG-NumberZeros-1; SizeS:=SizeS-1; NumberS:=NumberS-Low;#\\
\verb#      Offset:=Offset+NumberZeros+1;#\\
\verb#    fi;#\\
\verb#  od;#\\
\verb#end;#\hrule
\end{algo}

The differences to the subgroup search algorithm are: First we compute the set $W$ of all possible values of $|S|$, $|T|$ and $|U|$, that means $W:=\{2,\ldots,|G|-1\}^3$. Then we use all observations from section \ref{sec:concepts} to reduce the search space. Now we order the elements of $W$ decreasing by their multiplicative size (we use this to stop the algorithm in the case that we found a TPP triple, because it is the biggest TPP triple of $G$). In the outer \texttt{for} loop we go through all $i\in W$ and generate in the inner loops the $\binom{|G|-1}{i_1-1}$ subsets $S$ that contain the identity of $G$ and the $U$'s and $T$'s in the same way.

\begin{verbatim}
for i in W do
  for NrS in [1..Binomial(Size(G)-1,i[1]-1)] do
    IndexS := GenerateSubset(Size(G),i[1],NrS);
    for NrT in [1..Binomial(Size(G)-1,i[2]-1)] do
      IndexT := GenerateSubset(Size(G),i[2],NrT);
      if Size(Intersection(IndexS,IndexT)) = 1 then
        for NrU in [1..Binomial(Size(G)-1,i[3]-1)] do
          IndexU := GenerateSubset(Size(G),i[3],NrU);
          if (Size(Intersection(IndexS,IndexU)) = 1 
             and Size(Intersection(IndexT,IndexU)) = 1) then
            S := AsList(G){IndexS}; ...
\end{verbatim}
Of course one has to use a TPP test for subsets (see the previous section).

\section{Selected Results}

\noindent In this section we present some computational results of the brute-force search with subsets and subgroups. The values in the following tables are:
\begin{itemize}
\item \textsf{GAP} Id: The \texttt{IdSmallGroup} in the SmallGroups Library.
\item $D_3(G)$ resp. $D_3$: The $3$-character capacity (see definition \ref{def:DG}).
\item $\beta(G)$ resp. $\beta$: The TPP capacity (see definition \ref{def:TPPcap}).
\item $\beta_\mathrm{g}(G)$ resp. $\beta_\mathrm{g}$: The TPP subgroup capacity (see definition \ref{def:TPPcap}).
\item $\beta/|G|$ resp. $\beta_\mathrm{g}/|G|$: The so called \emph{TPP ratio} or \emph{TPP subgroup ratio}, resp.: A parameter that measures the size of the realized problem in relation to the group size.
\item $\tens{n, p, m}$: The parameters of the biggest realized problem.
\end{itemize}

\subsection{Computational Results for some Small Groups}
\noindent Our search algorithms allow us to compute all the relevant information 
about TPP triples of subsets for all nonabelian groups of order less than 
$25$. The results are shown in table \ref{tab:small} (computing TPP triples of subsets 
in groups of order greater than $24$ is too time consuming at present).

\begin{table}
\centering
\begin{tabular}{|l|l|l|l|l|l|l|l|}
\hline
\bf \textsf{GAP} Id &\bf  Group  & $\boldsymbol{D_3(G)}$ & $\boldsymbol{\beta(G)}$ & $\boldsymbol{\beta_\mathrm{g}(G)}$  & $\boldsymbol{\beta/D_3}$ & $\boldsymbol{\beta/|G|}$ & $\boldsymbol{\tens{n,p,m}}$\\\hline\hline
$[6, 1]$ 	& $S_3$  							& 10 	& 8 	& 8 			& 0.8 		& 1.33333 	&$2,2,2$\\\hline
$[ 8, 3 ]$ 	& $D_8$  							& 12 	& 8 	& 8 			& 0.666667 	& 1 		&$2,2,2$\\
$[ 8, 4 ]$ 	& $Q_8$ 							& 12 	& 8 	& 8 			& 0.666667 	& 1 		&$2,2,2$\\\hline
$[ 10, 1 ]$	&	$D_{10}$ 						& 18 	& \bf 12 	& \bf 10 			& 0.666667 	& 1.2 		&$3,2,2$\\\hline
$[12,1]$ 	& $C_3 \rtimes C_4$ 				& 20 	& \bf 16 	& \bf 12 		& 0.8 		& 1.33333 	&$4,2,2$\\
$[12,3]$ 	& $A_4$ 							& 30 	& 18 	& 18 	 		& 0.6 		& 1.5 		&$3,3,2$\\
$[12,4]$  	& $D_{12}$ 							& 20 	& 16 	& 16 	 		& 0.8 		& 1.33333 	&$4,2,2$\\\hline
$[14,1]$ 	& $D_{14}$ 							& 26 	& \bf 16 	& \bf 14 	 	& 0.615385 	& 1.14286 	&$4,2,2$\\\hline
$[16,3]$ 	& $(C_4 \times C_2)\rtimes C_2$ 	& 24 	& 16 	& 16 	 	& 0.666667 	& 1 		&$4,2,2$\\
$[16,4]$ 	& $C_4 \rtimes C_4 $				& 24 	& 16 	& 16 	 		& 0.666667 	& 1 		&$4,2,2$\\
$[16,6]$ 	& $C_8 \rtimes C_2 $				& 24 	& 16 	& 16 	 		& 0.666667 	& 1 		&$4,2,2$\\
$[16,7]$ 	& $D_{16}$							& 28 	& \bf 20 	& \bf 16 	 		& 0.714286 	& 1.25 		&$5,2,2$\\
$[16,8]$ 	& $QD_{16}$							& 28 	& 16 	& 16 	 		& 0.571429 	& 1 		&$4,2,2$\\
$[16,9]$ 	& $Q_{16}$							& 28 	& 16 	& 16 	 		& 0.571429 	& 1 		&$4,2,2$\\
$[16,11]$ 	& $C_2 \times D_8$					& 24 	& 16 	& 16 	 		& 0.666667 	& 1 		&$4,2,2$\\
$[16,12]$ 	& $C_2 \times Q_8$					& 24 	& 16 	& 16 	 		& 0.666667 	& 1 		&$4,2,2$\\
$[16,13]$ 	& $(C_4 \times C_2)\rtimes C_2$		& 24 	& 16 	& 16 	 		& 0.666667 	& 1 		&$4,2,2$\\\hline
$[18,1]$ 	& $D_{18}$							& 34 	& 24 	& 24 	 		& 0.705882 	& 1.33333 	&$6,2,2$\\
$[18,3]$ 	& $C_3 \times S_3$					& 30 	& 24 	& 24 	 		& 0.8 		& 1.33333 	&$6,2,2$\\
$[18,4]$ 	& $C_3^2\rtimes C_2$				& 34 	& 24 	& 24 	 		& 0.705882 	& 1.33333 	&$6,2,2$\\\hline
$[20,1]$ 	& $C_5 \rtimes C_4$					& 36 	& \bf 24 	& \bf 20 			& 0.666667 	& 1.2 		&$6,2,2$\\
$[20,3]$ 	& $C_5 \rtimes C_4$					& 68 	& 32 	& 32 	 		& 0.470588 	& 1.6 		&$4,4,2$\\
$[20,4]$ 	& $D_{20}$							& 36 	& \bf 24 	& \bf 20 			& 0.666667 	& 1.2 		&$6,2,2$\\\hline
$[21,1]$ 	& $C_7 \rtimes C_3$					& 57 	& 27 	& 27 	 		& 0.473684 	& 1.28571 	&$3,3,3$\\\hline
$[22,1]$ 	& $D_{22}$							& 42 	& \bf 28 	& \bf 22 	 	& 0.666667 	& 1.27273 	&$7,2,2$\\\hline
$[24,1]$ 	& $C_3 \rtimes C_8$					& 40 	& \bf 32 	& \bf 24 	 	& 0.8 		& 1.33333 	&$4,4,2$\\
$[24,3]$ 	& $\SL_2\ff_3$						& 54 	& 36 	& 36 	 		& 0.666667 	& 1.5 		&$4,3,3$\\
$[24,4]$ 	& $C_3 \rtimes  Q_8$				& 44 	& \bf 32 	& \bf 24	 	& 0.727273 	& 1.33333 	&$4,4,2$\\
$[24,5]$ 	& $C_4 \times  S_3$					& 40 	& 32	& 32 	 		& 0.8 		& 1.33333 	&$4,4,2$\\
$[24,6]$ 	& $D_{24}$							& 44 	& 32 	& 32	 		& 0.727273 	& 1.33333 	&$4,4,2$\\
$[24,7]$ 	& $C_2 \times (C_3 \rtimes C_4)$	& 40 	& \bf 32 	& \bf 24	 	& 0.8 		& 1.33333 	&$4,4,2$\\
$[24,8]$ 	& $(C_6 \times C_2) \rtimes C_2$	& 44 	& 32 	& 32	 		& 0.727273 	& 1.33333 	&$4,4,2$\\
$[24,10]$ 	& $C_3 \times D_8$					& 36 	& 24 	& 24	 		& 0.666667 	& 1 		&$6,2,2$\\
$[24,11]$ 	& $C_3 \times Q_8$					& 36 	& 24 	& 24 	 		& 0.666667 	& 1 		&$6,2,2$\\
$[24,12]$ 	& $S_4$								& 64 	& 36 	& 36 	 		& 0.5625 	& 1.5 		&$4,3,3$\\
$[24,13]$ 	& $C_2 \times A_4$					& 60 	& 36 	& 36 	 		& 0.6 		& 1.5 		&$4,3,3$\\
$[24,14]$ 	& $C_2 \times C_2 \times S_3$		& 40 	& 32 	& 32 	 		& 0.8 		& 1.33333 	&$4,4,2$\\\hline
\end{tabular}
\caption{Computational results for all nonabelian groups of order less than 25. (If $\beta$ and $\beta_\mathrm{g}$ differ, the corresponding values are printed bold.)}
\label{tab:small}
\end{table}

\subsection{Computational Results for some $2$-Groups}
\noindent An interesting point is, that the only groups we found with the brute-force search for subgroups that achieve $\beta /D_3 > 0.8$ are $2$-groups. The results for groups of order $64$ and $128$ are shown in table \ref{tab:Two}. (The results for all nonabelian groups of order $256$ are still in computation.)  Details about the subgroup search can be found in subsection \ref{subsec:further}.

\begin{table}
\centering
\begin{tabular}{|l|l|l|l|l|l|l|}
\hline
\bf \textsf{GAP} Id & \bf  Group  & $\boldsymbol{D_3(G)}$ & $\boldsymbol{\beta_\mathrm{g}(G)}$  & $\boldsymbol{\beta_\mathrm{g}/D_3}$ & $\boldsymbol{\beta_\mathrm{g}/|G|}$ & $\boldsymbol{\tens{n,p,m}}$\\\hline\hline
$[64,226]$ & $D_8^2$ & 144 & 128 & 0.888889 & 2 & $8,4,4$\\
\hline
$[128,29]$ & $(C_2 \times (C_8\rtimes C_2))\rtimes C_4$ & 304 & 256 & 0.842105 & 2 & $16,4,4$\\
$[128,1135]$ & $(C_2^3 \times D_8) \rtimes C_2$ & 304 & 256 & 0.842105 & 2 & $16,4,4$\\
$[128,1142]$ & $(C_2^3 \times D_8) \rtimes C_2$ & 304 & 256 & 0.842105 & 2 & $8,8,4$\\
$[128,1165]$ & $(C_2^3 \times D_8) \rtimes C_2$ & 304 & 256 & 0.842105 & 2 & $8,8,4$\\
$[128,2194]$ & $C_2 \times D_8^2$ & 288 & 256 & 0.888889 & 2 & $16,4,4$\\
$[128,2213]$ & $(C_2 \times C_4 \times D_8) \rtimes C_2$ & 288 & 256 & 0.888889 & 2 & $16,4,4$\\\hline
\end{tabular}
\caption{Computational results for selected $2$-groups.}
\label{tab:Two}
\end{table}

\subsection{Selected Results for $\SL_n\ff_q$ and $\PSL_2\ff_q$}
\noindent We also tested some groups of type $\SL_n\ff_q$ and $\PSL_2\ff_q$. The results are shown in the tables \ref{tab:PSL} and \ref{tab:SL}. These are groups that realize a relatively high \emph{TPP subgroup ratio} $\rho(G):=\beta(G)/|G|$. No other tested groups obtain ratios bigger than $4$.

Furthermore we present a theoretical result about a lower bound for the TPP capacity of $\SL_3\ff_2$:

\begin{lemma}
The finite special linear group $\SL_3\ff_2$ realizes $\tens{8,7,7}$ via a TPP triple of
subgroups, and its TPP capacity has the lower bound $\beta(\SL_3\ff_2)\geq 392 = \frac73|\SL_3\ff_2|$.
\end{lemma}

\begin{proof}
We write $G:=\SL_3\ff_2$ and define a proper subset $S\subset G$ by\[
S:=\left\{ \left[\protect\begin{array}{ccc}
1 & x & y\protect\\
0 & 1 & z\protect\\
0 & 0 & 1\protect\end{array}\right]\mbox{ }\Biggl|\mbox{ }x,y,z\in\mathbb{F}_{2}\right\} .\]
This is a subgroup of order $8$, consisting of all the upper unitriangular
matrices in $G$. Define two proper cyclic subgroups $T,U<G$ by \[
T:=\left\langle t_{0}:=\left[\protect\begin{array}{ccc}
1 & 1 & 1\protect\\
1 & 0 & 0\protect\\
1 & 1 & 0\protect\end{array}\right]\right\rangle ,\mbox{ }U:=\left\langle u_{0}:=\left[\protect\begin{array}{ccc}
0 & 0 & 1\protect\\
1 & 0 & 1\protect\\
1 & 1 & 1\protect\end{array}\right]\right\rangle .\]
The generators $t_{0}$ and $u_{0}$, as defined above, have order
$7$ in $G$, so $T$ and $U$ are cyclic of order $7$. We know that
$T\cap U=1$ (otherwise $t_{0}$ would generate
$U$, and we would have $T=U$, which is not true). In $TU:=\{t_{0}^{i}u_{0}^{j} : i,j\in\mathbb{Z}_{7}\}\subset G$ all nontrivial products $t_{0}^{i}u_{0}^{j}$, with $i,j\in\mathbb{Z}_{7}$
not both zero, are non-upper triangular matrices, and therefore $T,U,TU\subseteq (G\setminus S)\cup 1$.
Therefore, we have $S \cap TU=1$. Thus $(S,T,U)$ is a (basic) TPP triple of subgroups of $G$, and
by it $G$ realizes $\tens{8,7,7}$
of size $8\cdot 7\cdot 7=392$. The TPP capacity $\beta(G)$ is
at least its TPP capacity via subgroups $\beta_\mathrm{g}(G)$, which, by this result, is
at least $392$.
Since $\SL_3\ff_2=\PSL_3\ff_2=\PSL_2\ff_7$, this result is also true for these
latter groups.
\end{proof}

Note that it follows from the tables \ref{tab:PSL} and \ref{tab:SL} that $\beta_\mathrm{g}(\SL_3\ff_2)=392$ holds.

\begin{table}
\centering
\begin{tabular}{|l|r|r|r|l|l|l|}
\hline
\bf  Group  & $\boldsymbol{D_3(G)}$ &  $\boldsymbol{|G|}$ & $\boldsymbol{\beta_\mathrm{g}(G)}$  & $\boldsymbol{\beta_\mathrm{g}/D_3}$ & $\boldsymbol{\beta_\mathrm{g}/|G|}$ & $\boldsymbol{\tens{n,p,m}}$\\\hline\hline
$\PSL_2\ff_2$    & 10    & 6    & 8     & 0.8      & 1.33333 & $2,2,2$\\
$\PSL_2\ff_3$    & 30    & 12   & 18    & 0.6      & 1.5     & $3,3,2$\\
$\PSL_2\ff_4$    & 244   & 60   & 108   & 0.442623 & 1.8     & $12,3,3$\\
$\PSL_2\ff_5$    & 244   & 60   & 108   & 0.442623 & 1.8     & $12,3,3$\\
$\PSL_2\ff_7$    & 1126  & 168  & 392   & 0.348135 & 2.33333 & $8,7,7$\\
$\PSL_2\ff_8$    & 4072  & 504  & 1372  & 0.336935 & 2.72222 & $14,14,7$\\
$\PSL_2\ff_9$    & 3004  & 360  & 972   & 0.323569 & 2.7     & $12,9,9$\\
$\PSL_2\ff_{11}$ & 7038  & 660  & 1980  & 0.28133  & 3       & $55,6,6$\\
$\PSL_2\ff_{13}$ & 13556 & 1092 & 3276  & 0.241664 & 3       & $78,7,6$\\
$\PSL_2\ff_{17}$ & 40252 & 2448 & 10368 & 0.257577 & 4.23529 & $24,24,18$\\
$\PSL_2\ff_{19}$ & 63646 & 3420 & 14400 & 0.226251 & 4.21053 & $60,20,12$\\\hline
\end{tabular}
\caption{Computational results for selected groups of type $\PSL_2\ff_q$.}
\label{tab:PSL}
\end{table}
\begin{table}
\centering
\begin{tabular}{|l|r|r|r|l|l|l|}
\hline
\bf  Group  & $\boldsymbol{D_3(G)}$ &  $\boldsymbol{|G|}$ & $\boldsymbol{\beta_\mathrm{g}(G)}$  & $\boldsymbol{\beta_\mathrm{g}/D_3}$ & $\boldsymbol{\beta_\mathrm{g}/|G|}$ & $\boldsymbol{\tens{n,p,m}}$\\\hline\hline
$\SL_2\ff_2$ & 10   & 6   & 8    & 0.8      & 1.33333 & $2,2,2$\\
$\SL_2\ff_3$ & 54   & 24  & 36   & 0.666667 & 1.5     & $4,3,3$\\
$\SL_2\ff_4$ & 244  & 60  & 108  & 0.442623 & 1.8     & $12,3,3$\\
$\SL_2\ff_5$ & 540  & 120 & 216  & 0.4      & 1.8     & $24,3,3$\\
$\SL_2\ff_7$ & 2198 & 336 & 784  & 0.356688 & 2.33333 & $16,7,7$\\
$\SL_2\ff_8$ & 4072 & 504 & 1372 & 0.336935 & 2.72222 & $14,14,7$\\\hline
$\SL_3\ff_2$ & 1126 & 168 & 392  & 0.348135 & 2.33333 & $8,7,7$\\\hline
\end{tabular}
\caption{Computational results for selected groups of type $\SL_n\ff_q$.}
\label{tab:SL}
\end{table}

\subsection{Further Results}
\label{subsec:further}
\noindent At the moment many computations are still running on three supercomputers at the universities of Jena and Halle-Wittenberg. A first goal is to compute the TPP subgroup capacity $\beta_\mathrm{g}$ for all nonabelian groups of order up to 1000 (except 512 and 768). We also hope that there exists an algorithm that can effectively produce \emph{basic} subsets, because the computation of the TPP \emph{subset} capacity $\beta$ has a complexity of $\mathcal O(8^{|G|})$ without such an algorithm and it takes $8$ hours to compute $\beta$ for one group of order $24$.
All results (that means of the subset search for groups of order up to 24 and of the subgroup search of groups up to order 1000) can be found at \texttt{http://www2.informatik.uni-halle.de/da/hedtke/tpp/}. Note that some results are still missing, because they are in computation at the moment. 


\section*{Acknowledgements}
\noindent We would like to thank \textsc{David J. Green}, \textsc{Peter M. Neumann}, \textsc{Luke Oeding} and \textsc{Maxim Petrunin} for their helpful comments and suggestions. We also thank the Friedrich-Schiller-University Jena and the Martin-Luther-University Halle-Wittenberg which allowed us to use their supercomputers.

\bibliographystyle{amsalpha}

\end{document}